\documentclass[a4paper,twoside,11pt,leqno]{article}
\usepackage{fullpage}
\usepackage{amsmath} 
\usepackage{amssymb} \usepackage{amsopn}

\usepackage{amsthm}
\usepackage{amsfonts} \usepackage{mathrsfs}
\usepackage{mathabx}

\usepackage{verbatim}

\usepackage[OT2,T1]{fontenc}

\usepackage[pdfpagelabels,pdftex]{hyperref}
\usepackage[dvips]{graphicx} \usepackage[usenames]{color}
\usepackage[all]{xy}

\usepackage{setspace}
\usepackage{amssymb}
\usepackage{euscript}
\usepackage{cite}
\usepackage[all]{xy}
\usepackage{tabularx}

\theoremstyle{plain}
\usepackage{amsfonts}
\usepackage{graphicx}
\usepackage{amsmath}

\hypersetup{
  pdftitle={},
  pdfauthor={},
  pdfsubject={},
  pdfkeywords={},
  colorlinks=false,
  breaklinks=true,
  bookmarksopen=true,
  bookmarksnumbered=true,
  pdfpagemode=UseOutlines,
  plainpages=false}

\newcommand{\bF}{{\mathbb F}}

\newcommand{\bN}{{\mathbb N}}

\newcommand{\bQ}{{\mathbb Q}}

\newcommand{\bZ}{{\mathbb Z}}

\newcommand{\cG}{{\mathscr G}}

\newcommand{\cI}{{\mathscr I}}

\newcommand{\caO}{{\mathcal O}}

\newcommand{\fp}{{\mathfrak p}}
\newcommand{\fq}{{\mathfrak q}}

\DeclareSymbolFont{cyrletters}{OT2}{wncyr}{m}{n}
\DeclareMathSymbol{\Sha}{\mathalpha}{cyrletters}{"58}

\DeclareMathOperator{\Aut}{Aut}

\DeclareMathOperator{\coker}{coker}
\DeclareMathOperator{\im}{im}

\DeclareMathOperator{\Spec}{Spec}

\DeclareMathOperator{\Frob}{Frob}

\DeclareMathOperator{\res}{res}

\newcommand{\cd}{{\rm cd}}
\newcommand{\scd}{{\rm scd}}

\DeclareMathOperator{\Gal}{G}

\newcommand{\ab}{{\rm ab}}

\newcommand{\Syl}{{\rm Syl}}

\makeatletter
\newtheorem*{rep@theorem}{\rep@title}
\newcommand{\newreptheorem}[2]{%
\newenvironment{rep#1}[1]{%
 \def\rep@title{#2 \ref{##1}}%
 \begin{rep@theorem}}%
 {\end{rep@theorem}}}
\makeatother

\newtheorem{thm}{Theorem}[section]

\newtheorem{prop}[thm]{Proposition}

\newtheorem{cor}[thm]{Corollary}
\newtheorem{Cor}[thm]{Corollary}
\newtheorem{lm}[thm]{Lemma}

\newreptheorem{thm}{Theorem}
\newreptheorem{prop}{Proposition}
\newreptheorem{cor}{Corollary}

\theoremstyle{definition}
\newtheorem{Def}[thm]{Definition}

\newtheorem{rem}[thm]{Remark}

\newcommand{\abs}{\sharp}

\newcommand\rar{ \rightarrow }
\newcommand\tar{ \twoheadrightarrow }
\newcommand\har{ \hookrightarrow }
\newcommand\Rar{ \Rightarrow }
\newcommand\LRar{ \Leftrightarrow }

\newcommand\charac{\mathop{ \rm char}}
\newcommand\dirlim{\mathop{\underrightarrow{\lim} }}
\newcommand\prolim{\mathop{\underleftarrow{\lim} }}
\newcommand{\sm}{{\,\smallsetminus\,}}
\newcommand\N{\rm N}
\newcommand\Cl{\rm Cl}
\newcommand\coh{\rm H}
\newcommand\rank{\rm rk}

\newcommand\bap{\bar{\fp}}
\newcommand\baq{\bar{\fq}}

\newcounter{absatzcounter}[section]
\numberwithin{equation}{section}

\title{On some anabelian properties of arithmetic curves}
\author{Alexander Ivanov\thanks{The author was supported by the Mathematical Center Heidelberg} \thanks{ivanov@ma.tum.de}}

\begin{document}

\maketitle

\begin{abstract}
In this paper we generalize an argument of Neukirch from birational anabelian geometry to the case of arithmetic curves. In contrast to the function field case, it seems to be more complicate to describe the position of decomposition groups of points at the boundary of the scheme $\Spec \caO_{K,S}$, where $K$ is a number field and $S$ a set of primes of $K$, intrinsically in terms of the fundamental group. We prove that it is equivalent to give the following pieces of information additionally with the fundamental group $\pi_1(\Spec \caO_{K,S})$: the location of decomposition groups of boundary points inside it, the $p$-part of the cyclotomic character, the number of points on the boundary of all finite etale covers, etc. Under certain finiteness hypothesis on Tate-Shafarevich groups with divisible coefficients, one can reconstruct all this quantities from the fundamental group alone.

\end{abstract}

\section{Introduction}

Let $K$ be a number field, $\overline{K}$ a fixed algebraic closure and $\Gal_K$ the absolute Galois group. In \cite{Ne} Neukirch showed, using an argument involving the Brauer groups of $K$ and its extensions, that the group $\Gal_K$ determines intrinsically, how the decomposition groups $D_{\fp}$ of primes of $K$ lie inside it. In contrast to the case of curves over finite felds, which is now well-understood, in particular, due to Tamagawa \cite{Ta}, there is almost nothing known about the case of arithmetic curves. Not even this Brauer group argument of Neukirch generalizes from the birational to the arithmetic situation, i.e., if one replaces $\Spec K$ by $\Spec \caO_{K,S}$, where $S \supseteq S_{\infty}$ is a finite set of primes of $K$ and considers only the decomposition groups of primes in $S$. The reason for this failure is the obstruction given by the non-vanishing second Tate-Shafarevich group. We are interested in the question, how much information about the decomposition groups of primes in $S$ is encoded intrinsically in $\Gal_S := \pi_1(\Spec \caO_{K,S})$. It seems to be possible, at least using some additional information, to reconstruct the position of decomposition groups of primes in $S$ inside $\Gal_S$. Essentially, it turns out that it is equivalent to know one of the following data: the embeddings $D_{\bap} \har \Gal_S$ for $\bap \in S_f := S \sm S_{\infty}$; the cyclotomic character on $\Gal_S$; the $S$-class number of all finite subfields $K_S/L/K$; the number $\abs{S(L)}$ for all these $L$. The results of this paper are part of the Ph.D. thesis \cite{Iv} of the author.

\begin{thm}\label{thm:decequivcondsgeneral} Let $K$ be a number field, $S \supseteq S_{\infty}$ a finite set of primes. Assume at least two rational primes lie in $\caO_{K,S}^{\ast}$, and $p$ is one of them. Assume $(\Gal_S, p)$ are given. The knowledge of one of the following extra structures is equivalent to any other:
 \begin{itemize}
\item[(i)] The embeddings $\iota_{\bap} \colon D_{\bap} \har \Gal_S$ for $\bap \in S_f$.
\item[(ii)] The cyclotomic $p$-character $\chi_{p} \colon U \rar \bZ_{p}^{\ast}$ on some open $U \subseteq \Gal_S$.
\item[(iii)] For all open $U \subseteq \Gal_S$ with totally imaginary fixed field $L$, the group $\Cl_S(L)$.
\item[(iii)'] For all open $U \subseteq \Gal_S$ with totally imaginary fixed field $L$, the number $\abs{\Cl_S(L)/p}$.
\item[(iv)] For all open $U \subseteq \Gal_S$ with fixed field $L$, the number $\abs{S(L)}$.
\end{itemize}
Assume the decomposition subgroups at primes in $S_f$ are the full local groups. Then the knowledge of the above is also equivalent to the knowledge of the following:
\begin{itemize}
\item[(ii)'] The cyclotomic character on some open subgroup $U \subseteq \Gal_S$.
\end{itemize}
\end{thm}

Observe that in the arithmetic situation, to give $\Gal_S$ together with the cyclotomic character, corresponds in the geometric situation over a finite field, in some sense, to give the fundamental group of a curve together with the attached outer Galois representation. The assumption in the theorem, that there are at least two rational primes lying under $S$, implies by the work of Clozel and Chenevier \cite{CC} Theorem 5.1, that the decomposition groups in $\Gal_S$ of primes in $S_{p_1} \cup S_{p_2}$ are the full local groups. It does not imply in general that this holds for all primes in $S$ (but it still does for primes lying in the maximal subset of $S$, defined over a totally real subfield: cf. \cite{CC} Remark 5.3(i)).

\begin{rem}\label{rem:remb} We make the following observation. If one of the data in the theorem is determined with respect to an open subgroup $U_0 \subseteq \Gal_S$, then it is also determined for $\Gal_S$. Indeed, it is enough to see this for (i). So, if the embeddings into $U_0$ of the decomposition groups at $S_f$ inside $U_0$ are given, then (using Corollary \ref{cor:FKSnormalizerproperty}(ii) below) the whole projective system of continuous $\Gal_S$-sets $\prolim_{U \subseteq U_0, U \triangleleft \Gal_S} S_f(U)$ is determined (here, we write $S_f(U)$ for $S_f(L)$ if $L = K_S^U$), and one obtains the decomposition groups $D_{\bap} \subseteq \Gal_S$ as the stabilizers of points under the action of $\Gal_S$ on it.
\end{rem}

Of all the quantities listed in the theorem, the numbers $\abs{S(L)}$ seem to be the most accessible ones. First of all, the numbers $\abs{S_{\infty}(U)}$ are determined by $(\Gal_S,p)$ (cf. Proposition \ref{prop:recinv_underLeo}). Proposition \ref{prop:thenumbersSUschafin} given below, allow to reconstruct the numbers $\abs{S_f(U)}$ for all open $U \subseteq U_0 \subseteq \Gal_S$ with $U_0$ small enough, under certain finiteness assumption. Then Remark \ref{rem:remb} allows to reconstruct the numbers $\abs{S_f(U)}$ for all $U \subseteq \Gal_S$ open.

\begin{prop}\label{prop:thenumbersSUschafin} Let $K,S$ be a number field together with a set of primes. Let $p \in \caO_{K,S}^{\ast}$. Assume that $p$ is odd and $\mu_p \subset K$. Assume, the following holds: for any character $\chi \colon \Gal_S \rar \bZ_p^{\ast} = \Aut(\bQ_p/\bZ_p)$ whose restriction to $\frac{1}{p}\bZ/\bZ$ is trivial, the group $\Sha^2(\Gal_S, \bQ_p/\bZ_p(\chi))$ is finite. Then for any such $\chi$, the group $\coh^2(\Gal_S, \bQ_p/\bZ_p(\chi))$ is of finite corank and

\begin{equation}
\abs{S_f(K)} =  1 + \max\nolimits_{\chi} {\rm corank} (\coh^2(\Gal_S, \bQ_p/\bZ_p(\chi))).
\end{equation}
\end{prop}

This finiteness assumption on $\Sha^2(\Gal_S, \bQ_p/\bZ_p(\chi))$ has the following (at least partial) evidence: if $\chi$ is the cyclotomic $p$-character $\chi_p$, then $\Sha^2(\Gal_S, \bQ_p/\bZ_p(\chi)) = 0$ is easy to compute and for $\chi = \chi_p^{\otimes k}$ with $k \in \bZ \sm \{1\}$ Soul\'e showed in \cite{So} using K-theory, that $\coh^2(\Gal_S, \bQ_p/\bZ_p(\chi)) = 0$. Compare also \cite{NSW} 10.3.27 and the discussion preceeding it.

Directly from Theorem \ref{thm:decequivcondsgeneral} and Proposition \ref{prop:thenumbersSUschafin} we get:

\begin{Cor}[Local correspondence at the boundary]
For $i =1,2$, let $K_i, S_i$ be a number field together with a finite set of primes containing $S_{\infty}$. Assume that at least two rational primes lie completely under $S_i$, and assume that one of them, denoted $p$, lies under both. Let $\chi_{i,p}$ denote the $p$-cyclotomic character on $\Gal_{K_i,S_i}$. Let
\[ \sigma \colon \Gal_{K_1,S_1} \stackrel{\sim}{\longrightarrow} \Gal_{K_2,S_2} \]
\noindent be a topological isomorphism, such that $\chi_{2,p} \circ \sigma|_{U_1} = \chi_{1,p}|_{U_1}$ for some open subgroup $U_1 \subseteq \Gal_{K_1,S_1}$ holds. Then for any $\bap_1 \in S_f(K_{1,S_1})$, there is a unique prime $\sigma^{\ast}(\bap_1) \in S_f(K_{2,S_2})$, such that $\sigma(D_{\bap_1}) = D_{\sigma^{\ast}(\bap_1)}$. This defines a $\Gal_{K_1,S_1}$-equivariant bijection

\[ \sigma^{\ast} \colon S_{1,f}(K_{1,S_1}) \stackrel{\sim}{\longrightarrow} S_{2,f}(K_{2,S_2}), \]

\noindent which induces compatible bijections

\[ \sigma_{U_1}^{\ast} \colon S_{1,f}(L_1) \stackrel{\sim}{\longrightarrow} S_{2,f}(L_2), \]

\noindent for any $L_1/K_1$ finite with corresponding subgroup $U_1 \subseteq \Gal_{K_1,S_1}$ and $U_2 = \sigma(U_1)$ with corresponding field $L_2$. If the decomposition groups at primes in $S_{1,f}$ are the full local groups, then $\sigma_{U_1}$ preserves the residue characteristic and the absolute degree of primes.

Moreover, if $p$ is odd and if for $i = 1,2$, there is an open subgroup $U_i \subseteq \Gal_{K_i,S_i}$, such that for all characters $\chi \colon U_i \rar \bZ_p^{\ast}$ with torsion-free image, the group $\Sha^2(U_i, \bQ_p/\bZ_p(\chi))$ is finite, then the condition $\chi_{2,p} \circ \sigma|_{U_1} = \chi_{1,p}|_{U_1}$ is automatically satisfied.
\end{Cor}


\subsection*{Notation}

Our notation will essentially coincide with the notations in \cite{NSW}. We collect some of the most important notations here. For a pro-finite group $G$ we denote by $G(p)$ its maximal pro-$p$ quotient and by $G_p$ a $p$-Sylow subgroup. For a subgroup $H \subseteq G$, we denote by $\N_G(H)$ its normalizer in $G$.

For a Galois extension $M/L$ of fields, $\Gal_{M/L}$ denotes its Galois group. By $K$ we always denote an algebraic number field, that is a finite extension of $\bQ$. If $L/K$ is a Galois extension and $\bap$ is a prime of $L$, then $D_{\bap, L/K} \subseteq \Gal_{L/K}$ denotes the decomposition subgroup of $\bap$. If $\fp := \bap|_K$ is the restriction of $\bap$ to $K$, then we sometimes allow us to write $D_{\bap}$ or $D_{\fp}$ instead of $D_{\bap,L/K}$, if no ambiguity can occur. We write $\Sigma_K$ for the set of all primes of $K$ and $S,T$ will usually denote subsets of $\Sigma_K$. If $L/K$ is an extension and $S$ a set of primes of $K$, then we denote the pull-back of $S$ to $L$ by $S_L$, $S(L)$ or $S$ (if no ambiguity can occur). We write $K_S/K$ for the maximal extension of $K$, which is unramified outside $S$ and $\Gal_S := \Gal_{K,S}$ for its Galois group. Further, for $p \leq \infty$ a (archimedean or non-archimedean) prime of $\bQ$, $S_p = S_p(K)$ denotes the set of all primes of $K$ lying over $p$ and $S_f := S \sm S_{\infty}$. 

Let $K,S$ be a number field and a set of primes. Then $n_K, r_1(K),r_2(K)$ is the degree, the number of real and of conjugate pairs of complex embeddings of $K/\bQ$ and $\bN(S) := \bN \cap \caO_{K,S}^{\ast}$, i.e., $p \in \bN(S)$ if and only if $S_p \subseteq S$. Further, $\chi_p \colon \Gal_S \rar \bZ_p^{\ast}$ denotes the cyclotomic $p$-character for $p \in \bN(S)$ and $\Cl_S(K)$ the $S$-class group of $K$. If $U \subseteq \Gal_S$ is an open subgroup and $L = (K_S)^U$, then we sometimes write $\Cl_S(U),\abs{S(U)},$ etc. instead of $\Cl_S(L),\abs{S(L)}$, etc.

If (x),(y) are some sets of invariants of $K,S$ (like, for example, (i),(ii) in Theorem \ref{thm:decequivcondsgeneral}), then (x) $\rightsquigarrow$ (y) resp. (x) $\leftrightsquigarrow$ (y) will have the following meaning: if the data in (x) are known, then we can deduce the data in (y) from them resp. the knowledge of (x) and (y) is equivalent. In particular, (x) $\rightsquigarrow$ (y) implies that if two pairs $(K_i,S_i), i = 1,2$ are given, with $\Gal_{K_1,S_1} \cong \Gal_{K_2,S_2}$ and such that the data in (x) coincide for $i = 1,2$, then also the data in (y) are coincide.

A local field means always a non-archimedean local field.

\subsection*{Outline of the paper} 

In Section \ref{sec:FKS_for_S_finite} we study intersections of decomposition subgroups of different primes inside $\Gal_S$, which is the first step towards a proof of Theorem \ref{thm:decequivcondsgeneral}. Section \ref{sec:MTproof} is devoted to the proof of Theorem \ref{thm:decequivcondsgeneral}. In Section \ref{sec:discussion} we prove Proposition \ref{prop:thenumbersSUschafin} and discuss which invariants of $K,S$ can be recovered from $\Gal_S$ (plus possibly some further information). 


\subsection*{Acknowledgements}
The results in this paper coincide with a part of author's Ph.D. thesis \cite{Iv}, which was written under supervision of Jakob Stix at the University of Heidelberg. The author is very grateful to him for the very good supervision, and to Kay Wingberg, Johnannes Schmidt and a lot of other people for very helpful remarks and interesting discussions. The work on author's Ph.D. thesis was partially supported by Mathematical Center Heidelberg and the Mathematical Institute Heidelberg. Also the author is grateful to both of them for their hospitality and the excellent working conditions.


\section{Intersections of decomposition subgroups}\label{sec:FKS_for_S_finite}

Let $S$ be a finite set of primes of $K$ and let $\bap,\baq$ be two primes of $K_S$ lying over $S_f$. In this section we investigate, under the assumption that $S_p \cup S_{\infty} \subseteq S$, how big the intersection of the decomposition groups $D_{\bap}, D_{\baq}$ inside $\Gal_S$ is. If $S$ is the set of all primes of $K$, then this intersection is trivial by a theorem of F.K. Schmidt \cite{NSW} 12.1.3. Its proof does not generalize to the case of restricted ramification, so we use different arguments, all of which are simple applications of class field theory. The main result of this section, which will be used later in the text is Corollary \ref{cor:FKSnormalizerproperty}. It is an analog of \cite{NSW} 12.1.4 in the case of $\Gal_S$. Finally, in Section \ref{sec:interdecsatgoodprimes} we consider the case of primes $\bap,\baq$ not lying over $S$.


\subsection{Groups of p-decomposition type}\label{sec:gpsofpdectype}

One of the most frequently used objects in our investigations will be the $p$-Sylow subgroup of an absolute Galois group of a local field 
with residue characteristic $\neq p$. Such a group has a very special and easy structure: it is a non-abelian pro-$p$-Demushkin group of rank two. 
Recall (cf. \cite{NSW} 3.9.9) that a pro-$p$-group is called a Demushkin group, if $\coh^1(G,\bZ/p\bZ)$ is finite, $\coh^2(G,\bZ/p\bZ)$ is one-dimensional over $\bF_p$ and the cup-product from the first degree into the second is non-degenerate. To have a shortcut, we define:

\begin{Def}\label{def:pdecsubgroups}
A group of \emph{$p$-decomposition type} is a non-abelian pro-$p$ Demushkin group of rank $2$.
\end{Def}

By a theorem of Demushkin (cf. \cite{NSW} 3.9.11) a one-relator pro-$p$-group $G$ is a Demushkin group, if and only if for some integers $n \geq 1, q \geq 0$ (assume for simplicity that $q \neq 2$; for the case $q = 2$, cf. \cite{NSW} 3.9.19), $G$ is generated by $x_1, \dots, x_n$ subject to one relation:

\[ x_1^q(x_1,x_2)(x_3,x_4) \dots (x_{n-1},x_n) = 1, \]

\noindent where $(x,y) = x^{-1}y^{-1}xy$. The numbers $n,q$ are the rank and the invariant of $G$ respectively. For a group $G$ of $p$-decomposition type we have $n = 2$ and hence $q \neq 0$ (otherwise $G$ would be abelian). Thus $G$ is of the form $\bZ_p \ltimes \bZ_p$ with $\bZ_p \har \Aut(\bZ_p) = \bZ_p^{\ast}$ injective. We need a description of all closed subgroups of groups of $p$-decomposition type.

\begin{lm}\label{lm:pdectypelm}
Let $H$ be a group of $p$-decomposition type.
\begin{itemize}
\item[(i)]   A non-trivial closed subgroup of $H$ is either isomorphic to $\bZ_p$ or is of $p$-decomposition type.
\item[(ii)]  The open subgroups of $H$ are exactly the subgroups of $p$-decomposition type.
\item[(iii)] $H$ has a unique maximal closed normal pro-cyclic subgroup, denoted $H_n$. It is also the unique closed normal subgroup, such that $H/H_n$ is infinite pro-cyclic.
\item[(iv)]  If $N \subseteq H$ is open, then $N_n = N \cap H_n$.
\end{itemize}
\end{lm}

\begin{proof} Let $H$ be a group of $p$-decomposition type.

\noindent (i)+(ii): one verifies immediately that a closed subgroup $N \subseteq H$ is either $\cong \bZ_p$ or open. It remains to show that an open subgroup $N \subseteq H$ is of $p$-decomposition type. Let $H_n \triangleleft H$ be a closed normal subgroup of $H$, such that $H/H_n \cong \bZ_p$. One obtains an exact sequence:
\[ 1  \rar  N \cap H_n \rar N \rar N / N \cap H_n \rar 1 \]
with first and last term isomorphic to $\bZ_p$. Hence this sequence splits and $N \cong \bZ_p \ltimes_{\phi} \bZ_p$ for some $\phi \colon \bZ_p \rar \Aut(\bZ_p)$. Either the image of $\phi$ is $\{1\}$ or $\phi$ is injective. In the first case $N \cong \bZ_p \times \bZ_p$ and in the second $N$ is of $p$-decomposition type. The first case can not occur, as otherwise one would have the contradiction $ 3 = \scd_p(N) \leq \scd_p(H) = 2$.
(iii): Let $H_n$ be as above. Assume $\bZ_p \cong H_1 \triangleleft H$ is normal and $H_1 \not\subseteq H_n$. Then 
\[ H_1 / (H_1 \cap H_n) \har H/H_n \cong \bZ_p, \] 

\noindent i.e., $H_1 \cap H_n = 1$. Now $H_n$, $H_1$ are two normal subgroups of $H$ with trivial intersection, i.e., $H_n \times H_1 \subseteq H$. But $H_n \times H_1 \not\cong \bZ_p$ is not of $p$-decomposition type. This is a contradiction to (i). Hence $H_n$ is the unique maximal normal closed pro-cyclic subgroup of $H$.

Assume now, $H_2 \triangleleft H$ is normal with $H/H_2 \cong \bZ_p$ and $H_2 \not\supseteq H_n$. As $H_n/H_2 \cap H_n \har H/H_2 \cong \bZ_p$, we get $H_n \cap H_2 = 1$. The same reasoning as above gives a contradiction. Thus $H_2 \supseteq H_n$. Then $H_2 = H_n$ follows easily.

(iv): Since $N \subseteq H$ is open, $N \not\subseteq H_n$ and we have an inclusion $1 \neq N/N \cap H_n \har H/H_n$, hence $N/N \cap H_n$ is infinite pro-cyclic. Thus by (iii), $N \cap H_n = N_n$.
\end{proof}

To a group $H$ of $p$-decomposition type we can associate the character defining the semi-direct product:
\[\chi_H \colon H \tar H/H_n \har \bZ_p^{\ast} = \Aut(H_n).\]


\subsection{Local situation}\label{sec:cftapp_locsit}

\noindent Let $\kappa$ be a local field with residue characteristic $\ell$ and let $\Gal_{\kappa}$ be its absolute Galois group. For $p \neq \ell$, the $p$-Sylow subgroups of the maximal tame quotient 
\[ \Gal_{\kappa}^{\rm tr} \cong \hat{\bZ} \ltimes \hat{\bZ}^{(\ell^{\prime})} \]
of $\Gal_{\kappa}$ are of $p$-decomposition type, which can easily be seen directly. Consider now a $p$-Sylow subgroup $\Gal_{\kappa, p} \subseteq \Gal_{\kappa}$. The composition

\[ \Gal_{\kappa, p} \har \Gal_{\kappa} \tar \Gal_{\kappa}^{\rm tr} \]

\noindent is injective, since $p \neq \ell$ and the kernel of the second map is a pro-$\ell$-subgroup. Thus $\Gal_{\kappa, p}$ is isomorphic to a $p$-Sylow subgroup of $\Gal_{\kappa}^{\rm tr}$, and hence is of $p$-decomposition type.


\subsection{Metabelian covers}\label{sec:FKSarithMetabCase}

\begin{lm}\label{lm:CFTnaslujbe}
Let $K$ be a number field and $S \supseteq S_p \cup S_{\infty}$ a set of primes of $K$. Let $\bap \in (S_f \sm S_p)(K_S)$ and $\fp = \bap|_K$. Let $\cG_{\fp}$ denote the absolute Galois group of $K_{\fp}$ and $\cG_{\fp,p}$ a $p$-Sylow subgroup. Then the composition
\[ \phi \colon \cG_{\fp,p} \har \cG_{\fp} \tar D_{\bap} \har \Gal_S \]

\noindent is injective. In particular, any $p$-Sylow subgroup of $D_{\bap}$ is of $p$-decomposition type.
\end{lm}

\begin{proof}



\noindent Since $\fp \not\in S_p$, we have $\cG_{\fp,p} \cong (\cG_{\fp,p}/\cI_{\fp,p}) \ltimes \cI_{\fp,p}$, where both factors are isomorphic to $\bZ_p$ and the second is the inertia subgroup. Due to the cyclotomic $p$-extension, which realizes the maximal unramified $p$-extension at $\fp$ and is unramified outside $S_p \subseteq S$, the kernel of $\phi$ is contained in $\cI_{\fp,p}$. To show that $\ker(\phi) = 1$, it is enough to show that for any $n > 0$, there is a finite subextension of $K_S/K$, whose ramification degree at $\fp$ is $p^n$.

Therefore, let $L_0/K$ be the Hilbert class field of $K$ and set $L := L_0 K(\zeta_{p^n})$. This is an abelian extension of $K$, unramified outside $S_p$. The ideal $\fp$ is on the one side unramified in $L$, and on the other side principal (being principal already in $L_0$). Thus we can write

\[ \fp \caO_L = (\epsilon) = \fp_1 \fp_2 \dots \fp_r, \]

\noindent with $\epsilon \in \caO_L$, and $\fp_i$ unequal prime ideals of $\caO_L$. We can assume that $\bap|_L = \fp_1$. Since $\fp \in S$, we have $\epsilon \in \caO_{L,S}^{\ast}$, and the extension $L(\epsilon^{1/p^n})$ is unramified outside $S_p \cup S_{\fp} \subseteq S$. But since $\fp_1|\fp$ is unramified, we have
\[ v_{\fp_1}(\epsilon) = 1, \]

\noindent where $v_{\fp_1}$ denotes the valuation corresponding to $\fp_1$. Thus the local extension $L_{\fp_1}(\epsilon^{1/p^n})/L_{\fp_1}$ is tamely ramified of degree $p^n$. 
\end{proof}


\begin{prop}\label{prop:notopenintersec}
Let $\bap \neq \bar{\fq} \in S_f(K_S)$, such that there is a rational prime $p \in \caO^{\ast}_{K_S,S \sm \{\bap, \bar{\fq} \}}$. Choose some $p$-Sylow subgroups $D_{\bap,p} \subseteq D_{\bap}$ resp. $D_{\bar{\fq},p} \subseteq D_{\bar{\fq}}$. Then $D_{\bap,p} \cap D_{\bar{\fq},p}$ is not open in $D_{\bap,p}$. In particular, $D_{\bap} \cap D_{\bar{\fq}}$ is not open in $D_{\bap}$.
\end{prop}

\begin{proof} By Lemma \ref{lm:CFTnaslujbe}, $D_{\bap,p}$ resp. $D_{\bar{\fq},p}$ are groups of $p$-decomposition type. Let $\fp = \bap|_K$, $\fq = \bar{\fq}|_K$. By going up to a finite extension, we can assume $\fp \neq \fq$. Observe that the extension constructed in the proof of Lemma \ref{lm:CFTnaslujbe} is Galois and unramified in $\fq$, as $\fq \not\in S_p \cup \{ \fp \}$. Thus if $I_{\cdot,p} \subseteq D_{\cdot,p}$ denotes the corresponding inertia subgroup, we have $I_{\bap,p} \cap I_{\bar{\fq},p} = 1$.

Now, assume $D_{\bap,p} \cap D_{\bar{\fq},p} \subseteq D_{\bap,p}$ is open. The second group is of $p$-decomposition type, hence the first also is (Lemma \ref{lm:pdectypelm}(ii)). Hence, again by Lemma \ref{lm:pdectypelm}(ii), the inclusion $D_{\bap,p} \cap D_{\bar{\fq},p} \subseteq D_{\bar{\fq},p}$ is also open. The maximal normal pro-cyclic subgroup of $D_{\cdot, p}$ is $I_{\cdot,p}$. Thus by Lemma \ref{lm:pdectypelm}(iv) applied to both inclusions, the maximal normal pro-cyclic subgroup of $D_{\bap,p} \cap D_{\bar{\fq},p}$ is equal to $I_{\bap,p} \cap D_{\bar{\fq},p}$ and to $D_{\bap,p} \cap I_{\bar{\fq},p}$ simultaneously, i.e., these two intersections are equal. This implies $D_{\bap,p} \cap I_{\bar{\fq},p} = I_{\bap,p} \cap I_{\bar{\fq},p} = 1$. But this group, being the maximal normal pro-cyclic subgroup of a group of $p$-decomposition type must be isomorphic to $\bZ_p$. This is a contradiction.

Finally, if $D_{\bap} \cap D_{\bar{\fq}} \subseteq D_{\bap}$ were open, then also $D_{\bap,p} \cap D_{\bar{\fq}} \subseteq D_{\bap,p}$. But $D_{\bap,p} \cap D_{\bar{\fq}}$ is a pro-$p$-subgroup of $D_{\bar{\fq}}$, hence contained in a $p$-Sylow subgroup $D^{\prime}_{\bar{\fq},p}$ of it. Thus the intersection $D_{\bap,p} \cap D^{\prime}_{\bar{\fq},p} = D_{\bap,p} \cap D_{\bar{\fq}}$ would be open in $D_{\bap,p}$, which contradicts to the already proven part of the proposition.
\end{proof}

Observe that all arguments up to now made only use of solvable extensions of $K$, thus we could also replace $\Gal_S$ by its maximal solvable quotient. Before going on, we quote the following recent result of Clozel and Chenevier:

\begin{thm}[Clozel-Chenevier, \cite{CC} Theorem 5.1]\label{thm:Clozel_Chenevier} Let $S_f$ be a set of finite primes of $\bQ$. Let $S := S_f \cup \{\infty\}$. If $\abs{S_f} \geq 2$, then for any $p \in S$ and $\bap \in S(\bQ_S)$ lying over $p$, the map 

\[ \Gal_{\overline{\bQ_p}/\bQ_p} \tar D_{\bap,\bQ_S/\bQ} \subseteq \Gal_{\bQ,S} \]

\noindent is injective. 
\end{thm}

Its proof is rather involved and uses proven cases of the automorphic base change and results of Harris-Taylor on local Langlands correspondence. We also remark that to prove this result it is necessary to work with the full group $\Gal_{\bQ,S}$ and not with its maximal solvable quotient. We deduce an immediate corollary:

\begin{cor}\label{cor:CC_cor_real_locs}
Let $K$ be a number field, $p, \ell$ two different rational primes, $S$ a set of primes of $K$, such that $S \supseteq S_p \cup S_{\ell} \cup S_{\infty}$. Then for any $\fp \in S_p$ and $\bap \in S_p(K_S)$ lying over $\fp$, the map 

\[ \Gal_{\overline{K_{\fp}}/K_{\fp}} \tar D_{\bap,K_S/K} \subseteq \Gal_{K,S} \]

\noindent is injective. 
\end{cor}

\begin{proof}
Let $S_0 := \{p,\ell, \infty \}$. Then $S_0(K) \subseteq S$ and $\bQ_{S_0} \subseteq K_{S_0(K)} \subseteq K_S$. By Theorem \ref{thm:Clozel_Chenevier} of Clozel-Chenevier, $\bQ_{S_0}/\bQ$ realizes the maximal local extensions at $p$, i.e., for any extension $\bap_0$ of $p$ to $\bQ_{S_0}$, the field $\bQ_{S_0, \bap_0}$ is algebraically closed. hence for any $\bap \in S_p(K_S)$ with restriction $\bap_0$ to the subfield $\bQ_{S_0}$, the field $K_{S, \bap} \supseteq \bQ_{S_0,\bap_0}$ is also algebraically closed. This finishes the proof.
\end{proof}

Using this, we deduce from Proposition \ref{prop:notopenintersec} the following analog of \cite{NSW} 12.1.4 for $\Gal_S$:

\begin{cor}\label{cor:FKSnormalizerproperty}\mbox{}
\begin{itemize}
\item[(i)] If $p \in \caO_{K,S}^{\ast}, \bap \in (S_f \sm S_p)(K_S)$ and $H \subseteq D_{\bap}$ a closed subgroup, such that $H \cap D_{\bap,p} \subseteq D_{\bap,p}$ is open for some $p$-Sylow subgroup $D_{\bap,p} \subseteq D_{\bap}$, then $\N_{\Gal_S}(H) \subseteq D_{\bap}$.
\item[(ii)] Assume that at least two rational primes lie in $\caO_{K,S}^{\ast}$. Then the intersection of two distinct decomposition subgroups in $\Gal_S$ of primes in $S_f(K_S)$ is not open in any of them.
\end{itemize}
\end{cor}

\begin{proof}
(i): Let $x \in \N_{\Gal_S}(H)$. Then $H = xHx^{-1} \subseteq xD_{\bap}x^{-1} = D_{x\bap}$. Thus $D_{\bap} \cap D_{x\bap} \supseteq H$ contains an open subgroup of a $p$-Sylow subgroup of $D_{\bap}$. Proposition \ref{prop:notopenintersec} implies $x\bap = \bap$, i.e., $x \in D_{\bap}$.

(ii): By Proposition \ref{prop:notopenintersec}, the only case to consider, is $S_p \cup S_{\ell} \subseteq S$, $\bap \in S_p$, $\bar{\fq} \in S_{\ell}$ with $p \neq \ell$ (and there is no further prime to compare $D_{\bap}$ with $D_{\bar{\fq}}$). Assume $D_{\bap} \cap D_{\bar{\fq}} \subseteq D_{\bap}$ is open. By Corollary \ref{cor:CC_cor_real_locs}, $D_{\bap}$ resp. $D_{\baq}$ is isomorphic to the absolute Galois group of a $p$-adic resp. $\ell$-adic field. Hence also the open subgroup $D_{\bap} \cap D_{\bar{\fq}}$ of $D_{\bap}$ is isomorphic to a Galois group of a $p$-adic field. Hence $D_{\bap} \cap D_{\bar{\fq}}$ contains free pro-$p$-subgroups of any finite rank. But $D_{\bar{\fq}}$ does not, and we get a contradiction.
\end{proof}


\subsection{Intersection of decomposition subgroups at good primes} \label{sec:interdecsatgoodprimes}

Let $K$ be a number field and $S \supseteq S_p \cup S_{\infty}$ a finite set of primes. Arguments in this section make only use of abelian $p$-extensions, so we work with $\Gal_S^{\ab}(p)$ instead of $\Gal_S$. Let $K_S^{\ab}(p)$ denote the corresponding subfield of $K_S$. For short, we write $D_{\bap}$ for $D_{\bap, K_S^{\ab}(p)/K}$. We consider the intersections of decomposition subgroups at primes outside $S$.  Observe first, that if $\bap \in \Sigma_{K_S^{\ab}(p)} \sm S$, then we have natural surjections:

\[ \hat{\bZ} \tar D_{\bap} \tar \bZ_p. \]

\noindent Indeed, the first surjection holds, since $\bap|_K$ is unramified with finite residue field and the second due to the assumption on $S$ and the existence of the cyclotomic $p$-extension. We will use the infinite version of the Chebotarev density theorem to prove the following result. Let $\delta_K$ denote the Dirichlet density on $K$.

\begin{prop} Let $p$ be a rational prime, $S$ a finite set of primes of $K$ with $S_p \cup S_{\infty} \subseteq S$. Assume that $K$ is not totally real. Let $\bap \in \Sigma_{K_S^{\ab}(p)} \sm S$ and $\fp = \bap|_K$. Then there is a set $T_{\fp} \subseteq \Sigma_K \sm S$ with $\delta_K(T_{\fp}) = 1$, such that for all $\fq \in T_{\fp}$ and all extensions $\bar{\fq}$ of $\fq$ to $M$, the following holds:

\[ D_{\bap,p} \cap D_{\bar{\fq}, p} = 1. \]

\noindent In particular, the intersection of $D_{\bap}$ and $D_{\baq}$ is not open in each of them.
\end{prop}

\begin{proof} Since $K$ is not totally real, $r_2(K) \geq 1$ and hence $\rank_{\bZ_p} \Gal_S^{ab, p} \geq 2$ by \cite{NSW} 10.3.20. Let $H \cong \bZ_p^2$ be some quotient of $\Gal_S^{\ab}(p)$ with corresponding field $L \subseteq K_S^{\ab}(p)$, such that $\fp$ is not completely split in $L$ (such quotient exists due to the cyclotomic extension). Since $H$ is torsion-free, this implies that the composition $D_{\bap,p} \har \Gal_S^{\ab}(p) \tar H$ is injective, i.e., $D_{\bap,p} \tar D_{\bap, L/K}$ is an isomorphism.

We have $\bZ_p \cong D_{\bap,L/K} \subseteq H$. Consider $H \har H \otimes_{\bZ_p} \bQ_p$, and let $N := H \cap (D_{\bap,L/K} \otimes_{\bZ_p} \bQ_p)$, the intersection taken in $H \otimes_{\bZ_p} \bQ_p$. Then $N$ being compact and closed subgroup of $D_{\bap,L/K} \otimes_{\bZ_p} \bQ_p \cong \bQ_p$ is isomorphic to $\bZ_p$. Let $\mu$ be the Haar measure on $H$, such that $\mu(H) = 1$. Then $\mu(N) = 0$ and hence $\mu(H \sm N) = 1$ and $\mu(\partial(H \cap N)) = \mu(N) = 0$. By Chebotarev density theorem for infinite extensions, the set $T_{\fp}$ of primes of $K$, lying outside $S$, whose Frobenius lies in $H \sm N$ has density $1$, and thus satisfies the requirements of the proposition.
\end{proof}




\section{Modified argument of Neukirch} \label{sec:MTproof}

In this section we prove Theorem \ref{thm:decequivcondsgeneral}. Therefore we use a modification of Neukirch's argument involving Brauer groups (cf. \cite{Ne} Theorem 1). From now on until the end of this section, we permanently assume that $K$ is a number field, $S \supseteq S_{\infty}$ is a finite set of primes of $K$, that there are at least two rational primes under $S$ and that $p$ denotes one of them. 



\subsection{Local invariants}\label{sec:localanabprop}

For convenience we recall briefly the local situation. Local fields are not anabelian (cf. \cite{NSW} Remark before 12.2.7). This means that one can construct two non-isomorphic local fields $\kappa \not\cong \kappa^{\prime}$ with isomorphic absolute Galois groups: $\Gal_{\kappa} \cong \Gal_{\kappa^{\prime}}$. Nevertheless, the following invariants of $\kappa$ can be recovered from $\Gal_{\kappa}$: the characteristics $\charac{\kappa}$ of $\kappa$ and $\charac{\bar{\kappa}}$ of the residue field $\bar{\kappa}$, the cardinality $\abs{\bar{\kappa}}$ of $\bar{\kappa}$, the absolute degree $[\kappa: \bQ_p]$, if $\kappa$ is $p$-adic, the inertia and the wild inertia subgroups $V_{\kappa} \subset I_{\kappa} \subset \Gal_{\kappa}$, the Frobenius class $\Frob_{\kappa} \in \Gal_{\kappa}/I_{\kappa}$, the multiplicative group $\lambda^{\ast}$ of any finite extension $\lambda/\kappa$, the cyclotomic character $\chi_{\rm cycl}$ on $\Gal_{\kappa}$.

These invariants can be recovered using the cohomology with finite coefficients of $\Gal_{\kappa}$, the local reciprocity law and the structure of the tame quotient of $\Gal_{\kappa}$. This material is essentially covered by \cite{NSW}. Further we have a (reformulation of a) nice lemma, proven by Neukirch:

\begin{lm}[cf. \cite{Ne} Korollar 1]\label{lm:locfinindex}
Let $L, M$ be two local fields with $L$ $p$-adic, and assume an injection $\Gal_L \subseteq \Gal_M$ is given. Then $M$ is $p$-adic too, and $\Gal_L$ is of finite index in $\Gal_M$. Further $[M \colon \bQ_p] \leq [L \colon \bQ_p]$.
\end{lm}
\begin{proof} 
A proof can be found at the end of the proof of \cite{NSW} 12.1.9. 
\end{proof}


\subsection{Some lemmas} \label{sec:somelemmas}

\begin{lm}\label{lm:notcontained}
Let $p$ be a rational prime. Let $\Gal_{\kappa}$ be the absolute Galois group of a local field $\kappa$, $H \subset \Gal_{\kappa}$ a subgroup of $p$-decomposition type. Then $\kappa$ is not $p$-adic.
\end{lm}

\begin{proof}
Suppose $\kappa$ is $p$-adic. First, we choose some $H \subset U \subseteq \Gal_{\kappa}$ with last inclusion open, such that the image of $H$ in $U(p)$ is not (pro-)cyclic. Indeed, choose an open normal subgroup $V \triangleleft \Gal_{\kappa}$ such that $H/H \cap V$ is not (pro-)cyclic. Then let $U$ be the preimage under $\Gal_{\kappa} \tar \Gal_{\kappa}/V$ of the $p$-subgroup $H/H \cap V$.

Now, by \cite{NSW} 7.5.11, $U(p)$ is either free or a Demushkin group of rank $[\lambda : \bQ_p] + 2 > 2$, where $\lambda$ is the local field corresponding to $U$. In both cases
$U(p)$, being of finite cohomological dimension, is torsion-free, hence the image of $H$ in $U(p)$ is torsion-free, hence $H$ embeds into $U(p)$ (using Lemma \ref{lm:pdectypelm}, one sees that the kernel of the map $H \rar U(p)$ can only be the trivial subgroup of $H$). Now, $U(p)$ can neither be free: this contradicts $\cd_p H = 2$, nor a Demushkin group of rank $>2$: this contradicts Lemma \ref{lm:Demsubgroupsrankgrowth}. This finishes the proof.
\end{proof}

\begin{lm}\label{lm:Demsubgroupsrankgrowth}
Assume $H_m, H_n$ are two Demushkin pro-$p$-groups of ranks $m,n \geq 2$ respectively. If there is an inclusion $H_m \subseteq H_n$, then it is automatically open and $m = (H_n:H_m)(n-2) + 2$. In particular, $m \geq n$.
\end{lm}

\begin{proof}
If $H_m \subseteq H_n$ is open, then $m = (H_n : H_m)(n-2) + 2 \geq n$, which is well-known (cf. \cite{De} or \cite{An} for a purely group-theoretic proof). If $H_m \subseteq H_n$ is not open, then $p^{\infty}$ divides the index $(H_n : H_m)$ and \cite{NSW} Chap. III \S7 Ex.3 implies that $\cd_p H_m < \cd_p H_n$, which is absurd, since both numbers are equal to $2$.
\end{proof}

In the original proof Neukirch used the following fact: let $H \subseteq G_S$ be a closed subgroup, which is isomorphic to the absolute Galois group of a local field of characteristic $0$. If an open subgroup of $H$ is contained in a decomposition subgroup $D_{\bap}$ of a prime $\bap \in S$, then also $H \subseteq D_{\bap}$. Unfortunately, this easy fact can not be applied to Theorem \ref{thm:decequivcondsgeneral}, since we do not know in general, whether the groups $D_{\bap}$ are the full local groups for $\bap \in S_f$. However, a more precise treatement involving $p$-Sylow subrgoups of decomposition subgroups is aviable.

\begin{lm}\label{lm:pdectypelocal} Let $H \subseteq \Gal_S$ be a closed subgroup of $p$-decomposition type. Assume that there is an open subgroup $H_0$ of $H$ with $H_0 \subseteq D_{\bap}$ for some $\bap \in S_f$. Then $H \subseteq D_{\bap}$.
\end{lm}

\begin{proof}
Taking the intersection over all conjugates of $H_0$ in $H$, we can assume $H_0$ to be normal in $H$. By Lemma \ref{lm:pdectypelm}, $H_0$ is of $p$-decomposition type. Since two rational primes lie in $\caO_{K,S}^{\ast}$, the decomposition groups of primes in $S_p \subset S$ are the full local groups. Hence by Lemma \ref{lm:notcontained}, $\bap \not\in S_p$. Further, $H_0$ is a pro-$p$-subgroup of $D_{\bap}$, hence contained in a pro-$p$-Sylow subgroup $D_{\bap, p}$, which is again of $p$-decomposition type, since $\bap \not\in S_p$. Thus, $H_0 \subseteq D_{\bap, p}$ are both of $p$-decomposition type and the inclusion is open by Lemma \ref{lm:pdectypelm}. Since $H$ normalizes $H_0$, Corollary \ref{cor:FKSnormalizerproperty}(i) implies $H \subseteq D_{\bap}$.
\end{proof}

\subsection{Characterization of decomposition subgroups} \label{sec:ident_crits_for_dec_groups} 

Recall that in Section \ref{sec:gpsofpdectype} we associated to any group $H \cong \bZ_p \ltimes \bZ_p$ of $p$-decomposition type a character $\chi_H \colon H \rar \bZ_p^{\ast}$, which describes the action of the first $\bZ_p$ on the second. Recall that $\chi_p$ denotes the $p$-cyclotomic character on $\Gal_S$. For any open subgroup $U \subseteq G_S$, let $\pi_{p,U}$ denote the natural projection 
\begin{equation} \label{eq:pi_p_U}
\pi_{p,U} \colon U \tar \Cl_S(U)/p. 
\end{equation}

\noindent Then we have the following criteria for a subgroup of $p$-decomposition type of $G_S$ to lie in a decomposition subgroup of a prime.

\begin{prop}\label{prop:gencaseequivs} Let $H \subseteq \Gal_S$ be a closed subgroup of $p$-decomposition type. The following are equivalent:
\begin{itemize}
\item[(a)] $H \subseteq D_{\bap}$ for some $\bap \in S_f \sm S_p$.
\item[(b)] For some open subgroup $H_0 \subseteq H$, $\chi_p|_{H_0} = \chi_{H_0}$.
\end{itemize}
If moreover $\mu_p \subset K$, then they are also equivalent to
\begin{itemize}
\item[(c)] For $H$ the following condition holds:
\begin{itemize}
\item[(*)$_H$] For any $U \subseteq \Gal_S$ open: $H \subseteq U \Rightarrow H \subseteq \ker( \pi_{p,U} \colon U \tar \Cl_S(U)/p)$.
\end{itemize}
\end{itemize}
The prime $\bap$ in (a) is unique.
\end{prop}

\begin{proof} If $H \subseteq D_{\bap},D_{\bar{\fq}}$ with $\bap, \bar{\fq} \in S_f \sm S_p$, then $H \subseteq D_{\bap,p},D_{\bar{\fq},p}$ for some $p$-Sylow-subgroups, which are again of $p$-decomposition type. Hence by Lemma \ref{lm:pdectypelm}(ii), the last inclusions are open. Proposition \ref{prop:notopenintersec} implies then $\bap = \bar{\fq}$. This proves the uniqueness of $\bap$ in (a).

(a) $\Rightarrow$ (b): After replacing $\Gal_S$ by an appropriate open subgroup containing $H$, we can assume $H = D_{\bap,p} \cong \bZ_p \ltimes \bZ_p$ is a $p$-Sylow subgroup of $D_{\bap}$. Then the first $\bZ_p$ acts on the second as the unramified quotient on the inertia subgroup, i.e., by the $p$-cyclotomic character. This means $\chi_H = \chi_p|_H$.

(b) $\Rightarrow$ (a): By Lemma \ref{lm:pdectypelocal} we can assume that $K$ is totally imaginary. Again by Lemma \ref{lm:pdectypelocal} it is enough to show that $H_0 \subseteq D_{\bap}$ for some $\bap \in S_f$. First we claim that the restriction map

\[ \coh^2(\Gal_S, \mu_{p^{\infty}}) \rar \bigoplus_{\fp \in S(K)} \coh^2(D_{\fp}, \mu_{p^{\infty}}), \]

\noindent is injective. Indeed, its kernel is by Poitou-Tate duality equal to:

\begin{eqnarray*} 
\Sha^2(\Gal_S, \mu_{p^{\infty}}) &=& \dirlim_n \Sha^2(\Gal_S, \mu_{p^n}) = \dirlim_n [\Sha^1(\Gal_S, \bZ/p^n\bZ)^{\vee}] \\ 
&=& [\prolim_n \Sha^1(\Gal_S, \bZ/p^n\bZ)]^{\vee} \cong [\prolim_n (\Cl_S/p^n)^{\vee}]^{\vee} = 0,
\end{eqnarray*}

\noindent the last equality being true by finiteness of the Hilbert class field and as the transition maps in the inverse limit are multiplications by $p$. This proves our claim.

Now we can do the same for any open subgroup $U \subseteq \Gal_S$, and pass to the direct limit over all open $U$ containing $H_0$. Let $M$ denote the fixed field of $H_0$. By exactness of $\dirlim$ and some straightforward abstract nonsense we obtain:

\begin{equation}\label{eq:NeukirchH2modified}
0 \rar \coh^2(H_0, \mu_{p^{\infty}}) \rar \prod_{\fp \in S(M)} \coh^2(D_{\fp, K_S/M}, \mu_{p^{\infty}}).
\end{equation}

\noindent By (b), $\chi_p|_{H_0} = \chi_{H_0}$. Thus $\coh^2(H_0, \mu_{p^{\infty}}) = \bQ_p/\bZ_p$. From the sequence \eqref{eq:NeukirchH2modified}, there is a prime $\bap \in S_f$ with $\coh^2(D_{\bap, K_S/M}, \mu_{p^{\infty}}) \neq 0$. We claim that the prime $\fp = \bap|_M$ is indecomposed in $K_S/M$, i.e., that $H_0 = D_{\bap, K_S/M} \subseteq D_{\bap}$. Therefore, consider an open subgroup $H^{\prime} \subseteq H_0$ with corresponding fixed field $M^{\prime}$. For any open $H^{\prime} \subseteq U \subseteq \Gal_S$ with corresponding fixed field $L$, let $T_{\fp,H^{\prime}}(U)$ be the (finite) set of all primes of $L$ lying under a prime $\fp^{\prime} \in S_{\fp}(M^{\prime})$. Then we have the sequence

\[ \coh^2(U, \mu_{p^{\infty}}) \rar \bigoplus_{\fq \in T_{\fp,H^{\prime}}(U)} \coh^2(D_{\fq, K_S/L}, \mu_{p^{\infty}}) \rar 0, \]

\noindent which is exact by \cite{NSW} 9.2.1 (after passing to the limit over all finite submodules), since there are still non-archimedean primes in $S(L)$, which do not enter the index set of the direct sum. Passing to the limit over all open $U$ containing $H^{\prime}$ gives the exact sequence: 

\begin{equation}\label{eq:NeukirchH2modifiedsurj} \coh^2(H^{\prime}, \mu_{p^{\infty}}) \rar \bigoplus_{ \fp^{\prime} \in S_{\fp}(M^{\prime}) } \coh^2(D_{\fp^{\prime}, K_S/M^{\prime}}, \mu_{p^{\infty}}) \rar 0. \end{equation}

\noindent Since $\chi_p|_{H^{\prime}} = \chi_{H_0}|_{H^{\prime}} = \chi_{H^{\prime}}$, we have $\coh^2(H^{\prime}, \mu_{p^{\infty}}) \cong \bQ_p/\bZ_p$. Further, $\coh^2(D_{\fp^{\prime}, K_S/M^{\prime}}, \mu_{p^{\infty}}) \neq 0$. In fact, $D_{\fp^{\prime}, K_S/M^{\prime}}$ is conjugate to an open subgroup of $D_{\bap, K_S/M}$. But since $\coh^2(D_{\bap, K_S/M}, \mu_{p^{\infty}}) \neq 0$, also $\coh^2(V, \mu_{p^{\infty}}) \neq 0$ for any open subgroup $V \subseteq D_{\bap, K_S/M}$ (this is an easy fact on $p$-decomposition groups). By counting the coranks in  \eqref{eq:NeukirchH2modifiedsurj}, there is only one prime lying over $\fp$ in any finite extension $M^{\prime}/M$. Hence $\bap|_M$ is indecomposed.

Since $H$ is of $p$-decomposition type and the groups $D_{\bar{\fq}}$ with $\bar{\fq} \in S_p$ are the full local groups by Corollary \ref{cor:CC_cor_real_locs} (since two rational primes lie under $S$), Lemma \ref{lm:notcontained} implies $\bap \not\in S_p$.

(a) $\Rightarrow$ (c): Let $H \subseteq U \subseteq \Gal_S$ with last inclusion open. Consider the commutative diagram:

\centerline{
\begin{xy}
\xymatrix{ H \ar@^{(->}[r] & D_{\bap} \cap U \ar@^{(->}[r] \ar@{->>}[d] 	& U  \ar@{->>}[d] \ar@{->>}[rd]	&  \\
			   & (D_{\bap} \cap U)^{\ab} \ar[r]		& U^{\ab} \ar@{->>}[r]		& \Cl_S(U)/p.
}
\end{xy}
}

\noindent Since the composition of the maps in the lower row is zero by class field theory,
\[ H \subseteq D_{\bap} \cap U \subseteq \ker(U \tar \Cl_S(U)/p),\]

\noindent i.e., (*)$_H$ holds.

(c) $\Rightarrow$ (a): Assume now (*)$_H$ holds. For any $U \supseteq H$ open in $\Gal_S$ with corresponding field $L$, we have $\mu_p \subset L$, and hence by Poitou-Tate duality:
\[\Sha^2(U, \bZ/p\bZ) = \Sha^1(U, \mu_p)^{\vee} \cong \Sha^1(U, \bZ/p\bZ)^{\vee} = (\Cl_S(U)/p)^{\vee\vee} = \Cl_S(U)/p. \] 

\noindent This gives us the exact sequence:

\[  0 \rar \Cl_S(U)/p \rar \coh^2(U, \bZ/p\bZ) \rar \bigoplus_{\fp \in S(U)} \coh^2(D_{\fp, K_S/L}, \bZ/p\bZ). \]

\noindent Set $M = (K_S)^H$ and consider the limit of these sequences over all open $U \supseteq H$:

\[  0 \rar \dirlim_{H \subseteq U \subseteq \Gal_S} \Cl_S(U)/p \rar \coh^2(H, \bZ/p\bZ) \rar \prod_{\fp \in S(M)} \coh^2(D_{\fp, K_S/M}, \bZ/p\bZ), \]

\noindent This sequence is exact. We claim that $\dirlim_{H \subseteq U \subseteq \Gal_S} \Cl_S(U)/p = 0$. For an open $H \subseteq U \subseteq \Gal_S$, let $U^{\prime} := \ker(U \tar \Cl_S(U)/p)$. By the $S$-version of the principal ideal theorem, which states that any ideal class in $\Cl_S(U)/p$ gets trivial in the subfield of the Hilbert class field  corresponding to the quotient $\Cl(U) \tar \Cl_S(U)/p$ (cf. e.g. \cite{Ko} Theorem 8.11), the map $\Cl_S(U)/p \rar \Cl_S(U^{\prime})/p$, induced by inclusion on ideals, is zero. On the other side, $U^{\prime}$ appears in the index set of the limit due to (*)$_H$. Thus $\dirlim_{H \subseteq U \subseteq \Gal_S} \Cl_S(U)/p = 0$. Now we can conclude as in the (b) $\Rar$ (a) part (with $\mu_{p^{\infty}}$-coefficients replaced by $\bZ/p\bZ$), exactly as in the original argument of Neukirch \cite{Ne} Theorem 1.
\end{proof}

\begin{rem}\label{rem:crits_for_full_group}
With exactly the same proof (except for the uniqueness statements, which follow from Lemma \ref{lm:locfinindex} and Corollary \ref{cor:FKSnormalizerproperty}(ii) instead from Lemma \ref{lm:pdectypelocal} as above), the same criteria as in the proposition hold for $H$ if one assumes it to be a closed subgroup of $G_S$, which is isomorphic to the absolute Galois group of a local field of characteristic zero instead of a group of $p$-decomposition type.
\end{rem}


\subsection{Proof of Theorem \ref{thm:decequivcondsgeneral}}  \label{sec:gencaseproofeqcond}

\begin{proof}[Proof of (i) $\rightsquigarrow$ (ii)] Since we want to reconstruct the $p$-cyclotomic character $\chi_p$ only on an open subgroup of $\Gal_S$, we can assume $\mu_p \subset K$ and $K$ totally imaginary. Observe that $\chi_p$ on the local groups $D_{\bap}$ with $\bap \in S_p$ is determined by the group structure, since $D_{\bap}$ is the full local group in this case (cf. Section \ref{sec:localanabprop}). If $\bap \in S_f \sm S_p$, then $D_{\bap,p} \har D_{\bap} \tar D_{\bap}(p)$ is bijective; $\chi_p$ is determined on $D_{\bap,p}$ (in fact, it is equal to the character associated to the $p$-decomposition group $D_{\fp, p}$); and $\chi_p$ factors through $D_{\bap} \tar D_{\bap}(p)$. Thus $\chi_p$ is in this case also determined on $D_{\bap}$.  We have the following exact sequence from class field theory (\cite{NSW} 8.3.21(ii)):

\begin{equation}\label{eq:CFTSeqNF}
0 \rar \overline{\caO_{K,S}^{\ast}} \rar \prod_{\fp \in S(K)} D_{\bap}^{\rm ab} \rar \Gal_S^{\rm ab} \rar \Cl_S(K) \rar 0.
\end{equation}

\noindent The data given by (i) determine this sequence, since they determine the map in the middle. Since the global cyclotomic character factorizes through $\Gal_S^{\rm ab}$, it is determined by the local ones on the open subgroup $\ker(\Gal_S \tar \Cl_S(K))$ of $\Gal_S$.
\end{proof}

Under the additional assumption that the decomposition groups at $S_f$ are the full local groups, the proof of (i) $\rightsquigarrow$ (ii)' works similarly and (ii)' $\rightsquigarrow$ (ii) and (iii) $\rightsquigarrow$ (iii)' are immediate.

\begin{proof}[Proof of (i) $\rightsquigarrow$ (iii)] Assume the embeddings $(\iota_{\bap} \colon D_{\bap} \har \Gal_S)_{\bap \in S_f}$ are given. Then they are also given for any open subgroup $U \subseteq \Gal_S$. Let $U$ be such that the corresponding field $L$ is totally imaginary, i.e., the decomposition groups of archimedean primes are trivial. Then the sequence \eqref{eq:CFTSeqNF} for $U$ determines $\Cl_S(U)$ as the quotient of $U$ by the closure of the normal subgroup generated by the commutator and the images of $\iota_{\bap, L}$ for $\bap \in S_f$.
\end{proof}

\begin{proof}[Proof of (i) $\rightsquigarrow$ (iv)]
For any $U$, $\abs{S_f(U)}$ is equal to the number of the $U$-conjugacy classes of the subgroups $D_{\bap} \cap U$ and $\abs{S_{\infty}(U)}$ is given by the number of real/complex embeddings, which is deduced from $(\Gal_S,p)$ by Proposition \ref{prop:recinv_underLeo}.
\end{proof}

Finally we show the remaining directions, using criteria from Proposition \ref{prop:gencaseequivs}.

\begin{proof}[Proof of (ii) $\rightsquigarrow$ (i), (iii)' $\rightsquigarrow$ (i), (iv) $\rightsquigarrow$ (i)]
Assume (ii), (iii)' or (iv) is given. As we know that the decomposition subgroups of primes over $p$ are the full local groups and as the full local group determines the residue characteristic, Remark \ref{rem:crits_for_full_group} implies that we can reconstruct them from the given data.

\begin{lm}\label{lm:iiieqiv}
Assume $\mu_p \subset K$ (and $\mu_4 \subset K$ if $p = 2$) in Theorem \ref{thm:decequivcondsgeneral}. Then (iii)' $\leftrightsquigarrow$ (iv).
\end{lm}
\begin{proof} Since $\mu_p \subseteq K$, we have for every $U$ the exact sequence (${}_pA$ means the $p$-torsion of the abelian group $A$):
\[ 0 \rar \Cl_S(U)/p \rar \coh^2(U, \bZ/p\bZ) \rar  {}_p\coh^2(U, \caO_S^{\ast})  \rar 0, \text{ and }\]
\[ \dim_{\bF_p} {}_p \coh^2(U, \caO_S^{\ast}) = \abs{S_f(U)} - 1, \]

\noindent since $K$ is totally imaginary. Thus $\dim_{\bF_p} \coh^2(U, \bZ/p\bZ) + 1 = \dim_{\bF_p} \Cl_S(U)/p + \abs{S_f(U)}$. Since the number on the left is known, the knowledge of one of the summands on the right is equivalent to the knowledge of the other.
\end{proof}

\begin{lm} \label{lm:aus_iv_pi_p_U}
From the data in (iv) one can reconstruct the maps $\pi_{p,U}$ (cf. \eqref{eq:pi_p_U}) and for any $V \subseteq U \subseteq \Gal_S$ open, the maps $\Cl_S(U)/p \rar \Cl_S(V)/p$, which are induced by inclusion on ideals. 
\end{lm}
\begin{proof}
For any open $U$ with corresponding field $L$, we can describe the Galois group of the maximal abelian unramified extension of $L$, which is completely decomposed in $S$. By class field theory, it is canonically isomorphic to $\Cl_S(U)$. In fact, an extension of $L$, corresponding to an open subgroup $V \subseteq U$ is completely decomposed in $S$, if and only if $\abs{S(V)} = (U:V) \abs{S(U)}$. Observe that such extension is automatically unramified, since it is unramified outside $S$, as all groups are subquotients of $\Gal_S$, and also unramified in $S$, being completely decomposed there. Thus if we set $V_0 := \bigcap_V V$, where the intersection is taken over all open normal subgroups $V \subseteq U$, such that $\abs{S(V)} = (U:V) \abs{S(U)}$ and the quotient $U/V$ is abelian, then $U/V_0 \cong \Cl_S(U)$. Thus (iv) gives us the surjections $U \tar \Cl_S(U)$ and in particular the surjections
\[ \pi_{p,U} \colon U \tar \Cl_S(U)/p \]
\noindent (notice that (iii)' contains this information only implicitly!). Furthermore, for $V \subseteq U \subseteq \Gal_S$ open, the map $\Cl_S(U) \rar \Cl_S(V)$ induced by inclusion on ideals, is encoded in the group theory as the map induced by the transfer map $U^{\ab} \rar V^{\ab}$ (cf. e.g. \cite{Ne2}, after Proposition 6.13).
\end{proof}

Let now $U \subseteq \Gal_S$ be an open (normal) subgroup, small enough, such that the corresponding fixed field $L$ contains the $p$-roots of unity and is totally imaginary. By Proposition \ref{prop:gencaseequivs}, applied to $U$, using Corollary \ref{cor:FKSnormalizerproperty}(i) if necessary, we can decide, using the information given by (ii), (iii)' or (iv) and Lemmas \ref{lm:iiieqiv} and \ref{lm:aus_iv_pi_p_U}, whether a closed subgroup $H \subseteq U$ of $p$-decomposition type is contained in a decomposition subgroup of a prime in $S_f \sm S_p$. By Lemma \ref{lm:CFTnaslujbe} and Lemma \ref{lm:pdectypelocal}, the maximal subgroups with this property are exactly the $p$-Sylow subgroups of the groups $D_{\bap, K_S/L}$ with $\bap \in S_f \sm S_p$. Thus we have reconstructed the set
\[ \Syl_p(U, S_f \sm S_p) := \{H \subseteq U \colon H \text{ is a $p$-Sylow-subgroup of $D_{\bap,K_S/L}$ with } \bap \in S_f \sm S_p \}. \]

\noindent Now, $U$ acts on this set by conjugation. We have an $U$-equivariant surjection ($U$ acts trivially on the right side):

\[ \psi \colon \Syl_p(U, S_f \sm S_p) \tar (S_f \sm S_p)(U), \]

\noindent which sends $H$ to the (unique by Proposition \ref{prop:notopenintersec}!) prime $\bap|_L$, such that $H \subseteq D_{\bap,K_S/L}$. We want to determine, when two elements have the same image under $\psi$. For $H \in \Syl_p(U, S_f \sm S_p)$ such that $H \subseteq D_{\bap, K_S/L}$ is a $p$-Sylow subgroup, consider the restriction map

\[ \res^U_H \colon \coh^2(U, \bZ/p\bZ) \tar \coh^2(H, \bZ/p\bZ),\]

\noindent which is surjective, being equal to the composition
\[ \coh^2(U, \bZ/p\bZ) \tar \coh^2(D_{\bap,K_S/L}, \bZ/p\bZ) \stackrel{\sim}{\rar} \coh^2(H, \bZ/p\bZ), \]
\noindent in which the first map is surjective by \cite{NSW} 9.2.1, since $\abs{S_f(U)} > 1$, and the second is an isomorphism, since $\mu_p \subset L$.

\begin{lm} \label{lm:equivrelblabla}
Let $H, H^{\prime} \in \Syl_p(U, S_f \sm S_p)$. Then:
\[ \psi(H) = \psi(H^{\prime}) \LRar \ker(\res^U_H) = \ker(\res^U_{H^{\prime}}). \]
\end{lm}

\begin{proof} Consider the commutative diagram with exact row:

\centerline{
\begin{xy}
\xymatrix{
0 \ar[r] & \Sha^2(U, \bZ/p\bZ) \ar[r] & \coh^2(U, \bZ/p\bZ) \ar_{\res^U_H}[rd] \ar[r]	& (\bigoplus_{\fq \in S(L)} \coh^2(D_{\fq,K_S/L}, \bZ/p\bZ))^{\Sigma = 0} \ar[r] \ar[d] & 0\\
	 &  			   & 				     		& \coh^2(H, \bZ/p\bZ)
}
\end{xy}
}

\noindent where $\Sigma = 0$ means that we take the subspace of trace zero elements. The diagonal map factors through the vertical one, since $H \in \Syl_p(U, S_f \sm S_p)$. From this sequence we see, that if $\fp = \psi(H)$, then the kernel of $\res^U_H$ is the extension of the subspace $(\bigoplus_{\fq \in S(L) \sm \{\fp\} } \coh^2(D_{\fq,K_S/L}, \bZ/p\bZ))^{\Sigma = 0}$ of the space on the right side by $\Sha^2(U, \bZ/p\bZ)$. Two such subspaces of $\coh^2(U, \bZ/p\bZ)$ corresponding to $\fp$ resp. $\fp^{\prime}$ are equal if and only if $\fp = \fp^{\prime}$ (since we can assume $S_{p_1} \cup S_{p_2} \subsetneq S_f(U)$ and hence $\abs{S_f(U)} \geq 3$). This finishes the proof.
\end{proof}

The lemma gives a purely group-theoretical criterion to decide, whether two elements of $\Syl_p(U, S_f \sm S_p)$ lie in the same fibre of $\psi$. If we define an equivalence relation on  $\Syl_p(U, S_f \sm S_p)$ by $H \sim H^{\prime} :\LRar \ker(\res^U_H) = \ker(\res^U_{H^{\prime}})$, we get a bijective map induced by $\psi$:

\[\Syl_p(U, S_f \sm S_p)/\sim \quad \stackrel{\sim}{\longrightarrow} \quad (S_f \sm S_p)(U). \]

\noindent If $U^{\prime} \subseteq U \subseteq \Gal_S$, then we get a (non-canonical!) mapping
\[ \alpha \colon \Syl_p(U^{\prime}, S_f \sm S_p) \rar \Syl_p(U, S_f \sm S_p), \]

\noindent which sends $H^{\prime} \in \Syl_p(U^{\prime}, S_f \sm S_p)$ to some $H \in \Syl_p(U, S_f \sm S_p)$, such that $H^{\prime} \subseteq H$ (there is at least one by construction). If $H^{\prime} \subseteq H_1, H_2$, then $H_1, H_2 \subseteq D_{\bap}$ for some $\bap$ by Proposition \ref{prop:notopenintersec}. In particular, $\alpha$ induces a map

\[ \overline{\alpha} \colon \Syl_p(U^{\prime}, S_f \sm S_p)/\sim \rar \Syl_p(U, S_f \sm S_p)/\sim, \]

\noindent which is independent of the above choices. We obtain the following commutative diagram:

\centerline{
\begin{xy}
\xymatrix{
\Syl_p(U^{\prime}, S_f \sm S_p)/\sim  \ar^{\overline{\alpha}}[d] \ar^(0.55){\sim}[r] & (S_f \sm S_p)(U^{\prime}) \ar[d] \\
\Syl_p(U, S_f \sm S_p)/\sim  \ar^(0.55){\sim}[r]                 & (S_f \sm S_p)(U),
}
\end{xy}
}

\noindent where horizontal maps are bijections induced by $\psi$, and the vertical map on the right is the restriction of primes.

If $U \triangleleft \Gal_S$ is normal, then $\Gal_S$ acts on $\Syl_p(U, S_f \sm S_p)$ by conjugation. It is easy to see that this action induces via $\psi$ a $\Gal_S$-action on $(S_f \sm S_p)(U)$ and that this last action coincides with the action of $\Gal_S$ on this set by permuting the primes. \emph{In this way we have reconstructed the projective system of $\Gal_S$-sets $\{ (S_f \sm S_p)(U) \colon U \subseteq U_0, U \triangleleft \Gal_S \}$}, where $U_0 \subseteq \Gal_S$ is some open subgroup. Now the decomposition subgroups of primes in $S_f \sm S_p$ are exactly the stabilizers in $\Gal_S$ of elements in the $\Gal_S$-set $\prolim_{U \subseteq U_0, U \triangleleft \Gal_S} (S_f \sm S_p)(U)$. This finishes the proof of Theorem \ref{thm:decequivcondsgeneral}.
\end{proof}


\section{Invariants encoded in $\Gal_S$}\label{sec:discussion}

In this section we discuss easy consequences of Theorem \ref{thm:decequivcondsgeneral}, and prove Proposition \ref{prop:thenumbersSUschafin}.

\subsection{Recovering some global invariants}\label{sec:Leorec}

Let $K,S$ be a number field together with a finite set of primes. Assume there is a rational prime $p$ with $S \supseteq S_p \cup S_{\infty}$. Which invariants of $K$ are encoded in $\Gal_S$ resp. $(\Gal_S,p)$ resp. $(\Gal_S, p, \chi_p)$? The next two propositions determine some of these invariants.

\begin{prop}\label{prop:recinv_underLeo} Let $S$ be a finite set of primes of $K$. Assume there is a rational prime $p$ with $S_p \cup S_{\infty} \subseteq S$. Then $(G_S,p)  \rightsquigarrow [K:\bQ],r_1(K),r_2(K)$. If the Leopoldt conjecture is true for $K$ and for all rational primes, then $\Gal_S \rightsquigarrow [K:\bQ], r_1(K), r_2(K), \bN(S)$.
\end{prop}

\begin{proof} First we show the last statement. So, assume Leopoldt is true for $K$ and all rational primes. We show that $\Gal_S$ determines $r_2 = r_2(K)$ and the set $\bN(S)$. For any rational prime $p$ consider the number $r(p) := \rank_{\bZ_p} \Gal_S^{\ab}(p)$. The Leopoldt conjecture says that $r_2 + 1 = r(p)$ if $S_p \cup S_{\infty} \subseteq S$. If $S_p \not\subseteq S$, then at least the cyclotomic $\bZ_p$-extension is not contained in $K_S/K$, thus in this case
\[r(p) = \rank_{\bZ_p} \Gal_S^{\ab}(p) < \rank_{\bZ_p} \Gal_{S \cup S_p}^{\ab,p} = r_2 + 1.\]

\noindent Since $S_p \subseteq S$ for at least one $p$, we obtain $r_2 = \max_p \{r(p)\} - 1$, and a prime lies in $\bN(S)$ if and only if $r(p)$ is maximal.

Now it remains to recover $[K:\bQ]$ and $r_1$. Once $[K:\bQ]$ is known, $r_1$ can be recovered as $[K:\bQ] - 2r_2$. To recover $[K:\bQ]$, observe that if $K$ is totally imaginary, $[K:\bQ] = 2r_2$ can be recovered together with $r_2$. If $\pi : \Gal_S \tar \Gal_S^{\ab}$ denotes the natural surjection, and $U := \pi^{-1} (\im([(p-1)p] \colon \Gal_S^{\ab} \rar \Gal_S^{\ab}))$, then $U \subseteq \Gal_S$ is open and $L := K_S^U$ is totally imaginary. Indeed, $L$ contains the $p^2$-roots of unity, since they are contained in $K_S$ ($p^2$ and not simply $p$ is needed to cover the case $p=2$). Thus 
\[ [K:\bQ] = (\Gal_S : U)^{-1}[L:\bQ] = 2(\Gal_S : U)^{-1}r_2(L). \]

To show the first (unconditional) statement of the proposition, notice that once a prime $p \in \bN(S)$ is known, one obtains $r_2(K)$ as the negative of the Euler characteristic $- \chi(\Gal_S, \bZ/p\bZ)$ (\cite{NSW} 8.7.5) and $[K:\bQ], r_1(K)$ as above, without assuming Leopoldt.
\end{proof}

\begin{prop} Let $K,S$ be a number field together with a set of primes, s.t. the decomposition groups at primes in $S_f$ are the full local groups. Assume $\Gal_S$ is given together with any one (or, equivalently, all) pieces of information from Theorem \ref{thm:decequivcondsgeneral}. Then one can recover the following invariants of $K$ and its extensions:
\begin{itemize}
\item[(i)] For any $U \subseteq \Gal_S$ open with corresponding field totally imaginary, the class number $\Cl(U)$.
\item[(ii)] For every $U^{\prime} \subseteq U \subseteq \Gal_S$ open, with corresponding fields totally imaginary, the natural maps $\Cl(U) \rar \Cl(U^{\prime})$.
\item[(iii)] For $U \subseteq \Gal_S$ small enough, with $L = (K_S)^U$, the roots of unity $\mu(L)$.
\item[(iv)] For any $U \subseteq \Gal_S$ open with $L = (K_S)^U$, the absolute inertia and ramification degrees $f_{\fp, L/\bQ_{\ell}}$ and $e_{\fp, L/\bQ_{\ell}}$ of any $\fp \in S_f(L)$ ($\fp$ lies over $\ell$).
\item[(v)] The set $\bN(S)$.
\item[(vi)] The numbers $[K:\bQ], r_1(K), r_2(K)$.
\end{itemize}
\end{prop}

\begin{proof}
(i) + (ii): If $K$ is totally imaginary, one obtains the group $\Gal_{\emptyset} = \Gal_{K_{\emptyset}/K}$, as the quotient of $\Gal_S$ by the closure of the normal subgroup generated by the inertia subgroups of all $D_{\bap}$, $\bap \in S_f$. Then canonically $\Gal_{\emptyset}^{\ab} \cong \Cl(K)$. The maps between two class groups are given by the transfer maps in the class field theory.

(iii) follows from (i) $\leftrightsquigarrow$ (ii)$^{\prime}$ in Theorem \ref{thm:decequivcondsgeneral}.

(iv) follows from the anabelian properties of local fields listed in Section \ref{sec:localanabprop}.

(v): for any rational prime $\ell$, let $n(\ell) := \sum\limits_{\fp \in S \cap S_{\ell}} [K_{\fp} : \bQ_{\ell}]$. This number can be reconstructed from the given data. Thus, $\ell \in \bN(S) \LRar n(\ell)$ is maximal. Finally (vi) follows from (v) and Proposition \ref{prop:recinv_underLeo}.
\end{proof}



\subsection{The numbers $\abs{S_f(U)}$} \label{sec:numbers_SU}

\begin{proof}[Proof of Proposition \ref{prop:thenumbersSUschafin}] Recall that $\chi_p$ denotes the $p$-cyclotomic character, and that $\mu_p \subset K$ implies that its image lies in $\ker(\Aut(\bQ_p/\bZ_p) \tar \Aut(\frac{1}{p}\bZ/\bZ))$. Assume $\chi \colon \Gal_S \rar \bZ_p^{\ast}$ induces the trivial action on $\frac{1}{p}\bZ/\bZ$. We claim first that if $\chi|_{D_{\bap}} = \chi_p|_{D_{\bap}}$ for all $\bap \in S$, then $\chi = \chi_p$ on $\Gal_S$. Indeed, $\chi, \chi_p$ factor both through $\Gal_S^{\ab}$. Using sequence \eqref{eq:CFTSeqNF}, $\chi^{-1} \otimes \chi_p$ factors through a map $\Cl_S(K) \rar \bZ_p^{\ast}$, i.e., its image is finite, and on the other side the images of $\chi$ and $\chi_p$ lie in the subgroup $\ker(\Aut(\bQ_p/\bZ_p) \tar \Aut(\frac{1}{p}\bZ/\bZ)) \cong \bZ_p$, i.e., the image of $\chi^{-1} \otimes \chi_p$ does too, and hence is torsion-free. Thus $\chi^{-1} \otimes \chi_p$ is the trivial character of $\Gal_S$, or with other words $\chi = \chi_p$ on $\Gal_S$.

The last part of the Tate-Poitou sequence for the $\Gal_S$-modules $\bZ/p^n\bZ(\chi)$ gives, after changing to the limit over all $n > 0$, the following exact sequence:

\begin{equation*}
0 \rar \Sha^2(\Gal_S, \bQ_p/\bZ_p(\chi)) \rar \coh^2(\Gal_S, \bQ_p/\bZ_p(\chi)) \rar \bigoplus_{\fp \in S(K)} \coh^2(D_{\fp, K}, \bQ_p/\bZ_p(\chi)) \rar \coker \rar 0,
\end{equation*}

\noindent where
\begin{eqnarray*}
\coker  &=& \dirlim_n [\coh^0 (\Gal_S, \frac{1}{p^n}\bZ/\bZ(\chi^{-1} \otimes \chi_p))^{\vee}] = [\prolim_n \coh^0 (\Gal_S, \frac{1}{p^n}\bZ/\bZ(\chi^{-1} \otimes \chi_p))]^{\vee} = \\
        &=& [\coh^0 (\Gal_S, \bZ_p(\chi^{-1} \otimes \chi_p))]^{\vee} = \begin{cases} \bQ_p/\bZ_p & \text{if } \chi = \chi_p, \\ 0 & \text{if } \chi \neq \chi_p \end{cases}
\end{eqnarray*}

\noindent (the last equality holds, since the restriction map $\Aut(\bZ_p) \rar \Aut(p^n\bZ_p)$ is an isomorphism; thus if $\chi^{-1} \otimes \chi_p$ is trivial on some open subgroup of $\bZ_p$, then it is also trivial on $\bZ_p$). By our assumption, the corank (i.e., the $\bZ_p$-rank of the Pontrjagin-dual) of the first term in the sequence is zero. Thus the corank of the third term is equal to the sum of the coranks of the second and the last terms. We have the two cases:

\noindent \emph{Case $\chi = \chi_p$}. Then the corank of the third term is $\abs{S_f(K)}$ and the corank of the last term is $1$. Thus the corank of the second term is $\abs{S_f(K)} - 1$.

\noindent \emph{Case $\chi \neq \chi_p$}. Then by the claim above, $\chi|_{D_{\bap}} \neq \chi_p|_{D_{\bap}}$ for at least one $\bap \in S_f$. By Lemma \ref{lm:cyclocharH2}, the corank of the third term is $\leq \abs{S_f(K)}  - 1$, and the corank of the last term is $0$. Thus the corank of the second term is $\leq \abs{S_f(K)}  - 1$. The proposition follows.
\end{proof}

\begin{lm}\label{lm:cyclocharH2}
Let $\kappa$ be a local field, $p \neq \charac(\kappa)$ an odd prime. Let $\chi \colon \Gal_{\kappa} \rar \bZ_p^{\ast} = \Aut(\bQ_p/\bZ_p)$ be a character. The following are equivalent:
\begin{itemize}
\item[(i)]  $\coh^2(\Gal_{\kappa}, \bQ_p/\bZ_p(\chi)) \neq 0$.
\item[(ii)] $\chi$ is the $p$-part of the cyclotomic character.
\end{itemize}
\end{lm}

\begin{proof}[Proof of the lemma.]
Let $\chi_p$ denote the $p$-part of the cyclotomic character of $\Gal_{\kappa}$. The local duality gives:
\begin{eqnarray*}
 \coh^2(\Gal_{\kappa}, \bQ_p/\bZ_p(\chi)) &=& \dirlim_n \coh^2(\Gal_{\kappa}, \bZ/p^n\bZ(\chi)) = \dirlim_n [\coh^0(\Gal_{\kappa}, \bZ/p^n\bZ(\chi^{-1} \otimes \chi_p))^{\vee}] \\
                    &=& [\prolim_n \coh^0(\Gal_{\kappa}, \bZ/p^n\bZ(\chi^{-1} \otimes \chi_p))]^{\vee} = [\coh^0(\Gal_{\kappa}, \bZ_p(\chi^{-1} \otimes \chi_p))]^{\vee} \\
                    &=& \begin{cases} \bQ_p/\bZ_p & \text{if } \chi = \chi_p \\ 0 & \text{if } \chi \neq \chi_p. \end{cases}
\end{eqnarray*}

\noindent The last equality holds by the same reasoning as in the proposition.
\end{proof}

\begin{rem}
Observe that the proof of Proposition \ref{prop:thenumbersSUschafin} does not determine $\chi_p$ directly as \emph{the} character with the maximal corank of $\coh^2(\Gal_S, \bQ_p/\bZ_p(\chi))$, but only intrinsically by determining the numbers $\abs{S(U)}$ and using Theorem \ref{thm:decequivcondsgeneral}.
\end{rem}



\renewcommand{\refname}{References}

\end{document}